\documentclass{article}

\usepackage[english]{babel}

\usepackage[letterpaper,top=2cm,bottom=2cm,left=3cm,right=3cm,marginparwidth=1.75cm]{geometry}

\usepackage{amsmath, amssymb, amsthm}
\usepackage{mathtools}
\usepackage{graphicx}
\usepackage{tikz}
\usepackage{subcaption}
\usepackage[colorlinks=true, allcolors=blue]{hyperref}
\usepackage{caption}
\usepackage{empheq}
\usepackage{nicematrix}
\usepackage{tabularx}
\usepackage{array}
\usepackage{booktabs}
\usepackage{enumerate}

\usetikzlibrary{positioning, arrows.meta}

\tikzset{
  plotframe/.style={
    draw=black!45,
    line width=0.4pt
  },
  plotaxis/.style={
    -{Latex[length=3pt,width=2.2pt]},
    line width=0.4pt
  },
  plottick/.style={
    line width=0.4pt
  },
  plotpoint/.style={
    fill=black
  }
}

\theoremstyle{plain}
\newtheorem{theorem}{Theorem}[section]
\newtheorem{lemma}[theorem]{Lemma}
\newtheorem{proposition}[theorem]{Proposition}

\theoremstyle{definition}
\newtheorem{definition}[theorem]{Definition}

\theoremstyle{remark}
\newtheorem{remark}[theorem]{Remark}

\numberwithin{figure}{section}
\newlength{\panelwidth}
\newcommand{\commonaxes}{%
  \path[use as bounding box] (-1.02,-1.02) rectangle (1.02,1.02);

  \draw[plotframe] (-1,-1) rectangle (1,1);

  \draw[plotaxis]
    (-0.82,0) -- (0.82,0)
    node[right=1pt] {$x$};

  \draw[plotaxis]
    (0,-0.82) -- (0,0.82)
    node[above=1pt] {$y$};

  \draw[plottick] (-0.5,-0.025) -- (-0.5,0.025);
  \draw[plottick] ( 0.5,-0.025) -- ( 0.5,0.025);

  \node[below=3pt] at (-0.5,0) {$-\frac12$};
  \node[below=3pt] at ( 0.5,0) {$ \frac12$};

  \draw[plottick] (-0.025,-0.5) -- (0.025,-0.5);
  \draw[plottick] (-0.025, 0.5) -- (0.025, 0.5);

  \node[left=3pt] at (0,-0.5) {$-\frac12$};
  \node[left=3pt] at (0, 0.5) {$ \frac12$};
}

\title{Classification of Stationary Configurations of Four Identical Point Vortices}
\author{Woojin Kim}

\begin{document}
\maketitle

\begin{abstract}
We find and classify every stationary configuration of the system of four identical point vortices in the plane.
Up to the natural symmetries of the system, there are exactly three configurations, which are relative equilibria: a square, a collinear configuration, and an equilateral triangle with a central vortex.
Among them, the square configuration is the global minimizer of the Hamiltonian, whereas the other two configurations are saddle points.
In particular, the square is the unique stable configuration among the stationary configurations.
Using the special algebraic structure of four points, we introduce a change of variables based on the Hadamard matrix of order four, which transforms the original four-vortex problem into an equivalent three-point problem.

\vspace{0.8em}
\noindent\textbf{Keywords:} Four-vortex system, Stationary Configuration, Relative equilibrium, Hadamard matrix.

\end{abstract}

\section{Introduction}

\subsection{Point-vortex dynamics and relative equilibria}

The point-vortex model is a classical model in fluid dynamics that describes the motion of idealized vortices in a two-dimensional incompressible and inviscid fluid.
The equations governing interacting point vortices originate in Helmholtz’s 1858 work on vortex motion and were subsequently formulated in Hamiltonian form by Kirchhoff. 

\begin{definition}
A \textit{configuration} is a set of four distinct points in the plane, denoted by $(z_1, z_2, z_3, z_4) \in \mathbb{C}^4$.
Here $z_i = x_i + i y_i$ represents the position $(x_i, y_i)$ of the $i$th vortex, for $i = 1, 2, 3, 4$.
\end{definition}

For a configuration $(z_1, z_2, z_3, z_4)$, let $\Gamma_i$ be the intensity of the $i$-th vortex at $z_i$.
Then the system has a Hamiltonian $W$ given by

\[
W(z_1, z_2, z_3, z_4) = - \frac{1}{2\pi} \sum_{1 \le i < j \le 4} \Gamma_i \Gamma_j \log |z_i - z_j|.
\]

The equations of motion of the system are given by the Kirchhoff equations:

\[
\Gamma_i \frac{d x_i}{dt} = \frac{\partial W}{\partial y_i}, \qquad \Gamma_i \frac{d y_i}{dt} = -\frac{\partial W}{\partial x_i}, \qquad i = 1, 2, 3, 4.
\]

In this paper, we consider the case of four identical vortices: $\Gamma_i = \Gamma$ for $1 \le i \le 4$.

The initial motivation for this problem came from Mayer's floating-magnet experiments in an external magnetic field.
Mayer inserted magnetized needles into corks so that they floated vertically on water, with all needles oriented in the same direction.
He then held a strong magnet above the water with the opposite pole facing the needles.
In this arrangement, the needles repel one another while the external magnet attracts them toward the center.
Mayer investigated the resulting equilibrium configurations, examples of which are shown in Figure \ref{fig:magnet_config}.

\begin{figure}[htbp]
    \centering
    \includegraphics[width=0.55\textwidth]{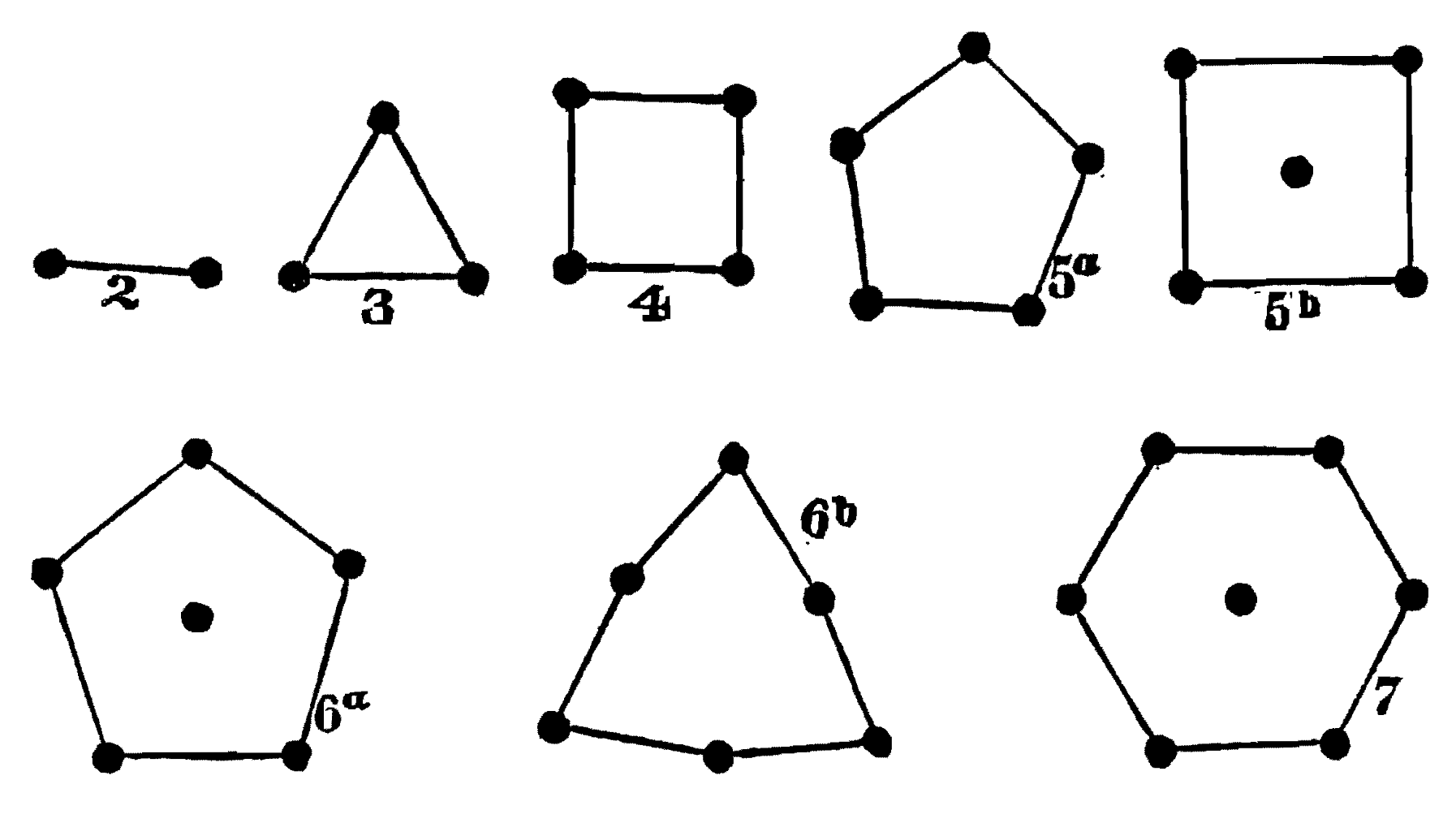}
    \caption{Selected configurations from Mayer’s floating-magnet experiments \cite{Mayer1878}.}
    \label{fig:magnet_config}
\end{figure}

This floating-magnet system is analogous to the point-vortex system considered by W. Thomson in his study of the stability of regular vortex $n$-gons.
The static equilibria of the floating magnets correspond to relative equilibria of the point-vortex system; point-vortex dynamics therefore provides a model for the behavior of the floating magnets.

We next define the center and impulse of a configuration, both of which are conserved quantities for the identical-vortex system.

\begin{definition}
The \textit{center} of the configuration $\bar{z} \in \mathbb{C}$ and the \textit{impulse} of the configuration $I \in \mathbb{R}^+$ are defined as
\[
\bar{z}= \frac{1}{4} \sum_{i=1}^4 z_i, \qquad
I = \sum_{i=1}^4 |z_i - \bar{z}|^2.
\]
\end{definition}

The main objects of study in this paper are stationary configurations and relative equilibria of the point-vortex system.
We follow the definitions introduced by O'Neil \cite{ONeil1987}.

\begin{definition} [\textbf{O'Neil} \cite{ONeil1987}]
A configuration $(z_1, z_2, z_3, z_4) \in \mathbb{C}^4$ is a \textit{stationary configuration} if there exists a constant $\omega \in \mathbb{C}$ so that the velocities and positions satisfy
\[
\dot{z_i} - \dot{z_j} = \omega(z_i - z_j), \quad 1 \le i < j \le 4.
\]
In particular, a configuration is a \textit{relative equilibrium} if there exist constants $0 \neq \omega \in \mathbb{R}, z_0 \in \mathbb{C}$, so that
\[
\dot{z_i} = i \omega (z_i - z_0), \quad 1 \le i \le 4.
\]
\end{definition}

The geometric meaning of a stationary configuration is a configuration that the shape of the arrangement of points is unchanging over time.

\subsection{Main result}

The central result of this paper is a purely analytical classification and complete variational characterization of stationary configurations of four identical point vortices.
Up to symmetries, there are only three kinds of stationary configurations, which are relative equilibria: a square, a collinear, and an equilateral triangle with a center.
Crucially, these configurations can be characterized by their energetic properties on the phase space constrained by the conserved quantities. 
For any given center and impulse, the square configuration is the global minimizer of the Hamiltonian. 
The other two configurations correspond to saddle points on this constrained energy landscape.

\begin{theorem}\label{maintheorem}

Let $(z_1(t), z_2(t), z_3(t), z_4(t)) \in \mathbb{C}^4$ be a configuration with center $z_c \in \mathbb{C}$, impulse $I \in \mathbb{R}^+$, and identical intensities $\Gamma \neq 0$. 
Let $\displaystyle{\omega = \frac{3\Gamma}{\pi I}}$. The configuration is a stationary configuration if and only if it belongs to one of the following three types, up to reindexing of vortices and a global spatial rotation about $z_c$:

\begin{enumerate}[\upshape 1.]
    \item \textbf{Square configuration (global minimizer):}
    \[
    z_j(t) = z_c + \frac{1}{2}\sqrt{I} e^{i(\omega t + \frac{\pi}{2}(j-1))}, \quad j = 1, 2, 3, 4.
    \]

    \item \textbf{Collinear configuration (saddle point):}
    \[
    \begin{aligned}
    z_j(t) &= z_c + \frac{1}{2\sqrt{3-\sqrt{6}}}\sqrt{I} e^{i(\omega t + \pi(j-1))}, \quad j = 1, 2, \\
    z_j(t) &= z_c + \frac{1}{2\sqrt{3+\sqrt{6}}}\sqrt{I} e^{i(\omega t + \pi(j-1))}, \quad j = 3, 4.
    \end{aligned}
    \]

    \item \textbf{Equilateral triangle with a center configuration (saddle point):}
    \[
    \begin{aligned}
    z_1(t) &= z_c, \\
    z_j(t) &= z_c + \frac{1}{\sqrt{3}}\sqrt{I} e^{i\left(\omega t + \frac{2\pi}{3}(j-2)\right)}, \quad j = 2, 3, 4.
    \end{aligned}
    \]
\end{enumerate}

On the constrained phase space with fixed $I$, the square configuration is the unique global minimizer of the Hamiltonian, whereas the collinear and equilateral triangle configurations are saddle points.

\end{theorem}

Figures \ref{fig:square}, \ref{fig:collinear}, and \ref{fig:triangle-center} illustrate the shapes of the three types of stationary configurations, respectively.
(assuming $\frac{I}{\Gamma} = 1$ for simplicity).
The explicit solutions in the Theorem \ref{maintheorem} describe motions in which each shape shown in the figures has a fixed center and rotates at a specific angular velocity determined by the configuration's size and the vortex intensity.

\par\medskip

{%
\setlength{\panelwidth}{0.31\linewidth}
\centering

\begin{minipage}[t]{\panelwidth}
  \centering

  \resizebox{\linewidth}{!}{%
    \begin{tikzpicture}[
      x=3cm,
      y=3cm,
      every node/.style={font=\scriptsize}
    ]
      \commonaxes

      \fill ( 0.5, 0) circle[radius=2.5pt];
      \fill (-0.5, 0) circle[radius=2.5pt];
      \fill (0,  0.5) circle[radius=2.5pt];
      \fill (0, -0.5) circle[radius=2.5pt];

    \end{tikzpicture}%
  }

  \captionof{figure}{Square configuration.}
  \label{fig:square}
\end{minipage}%
\hfill
\begin{minipage}[t]{\panelwidth}
  \centering

  \resizebox{\linewidth}{!}{%
    \begin{tikzpicture}[
      x=3cm,
      y=3cm,
      every node/.style={font=\scriptsize}
    ]
      \commonaxes

      \pgfmathsetmacro{\xone}{
        0.25*(
          sqrt(2 + 2/sqrt(3))
          +
          sqrt(2 - 2/sqrt(3))
        )
      }

      \pgfmathsetmacro{\xtwo}{
        0.25*(
          sqrt(2 + 2/sqrt(3))
          -
          sqrt(2 - 2/sqrt(3))
        )
      }

      \foreach \x in {\xone,\xtwo,-\xone,-\xtwo}{
        \fill[plotpoint] (\x,0) circle[radius=2.5pt];
      }
    \end{tikzpicture}%
  }

  \captionof{figure}{Collinear configuration.}
  \label{fig:collinear}
\end{minipage}%
\hfill
\begin{minipage}[t]{\panelwidth}
  \centering

  \resizebox{\linewidth}{!}{%
    \begin{tikzpicture}[
      x=3cm,
      y=3cm,
      every node/.style={font=\scriptsize}
    ]
      \commonaxes

      \pgfmathsetmacro{\leftx}{-1/(2*sqrt(3))}
      \pgfmathsetmacro{\rightx}{1/sqrt(3)}

      \fill[plotpoint] (\leftx, 0.5) circle[radius=2.5pt];
      \fill[plotpoint] (\leftx,-0.5) circle[radius=2.5pt];
      \fill[plotpoint] (\rightx,0) circle[radius=2.5pt];
      \fill[plotpoint] (0,0) circle[radius=2.5pt];
    \end{tikzpicture}%
  }

  \captionof{figure}{Equilateral triangle with a center configuration.}
  \label{fig:triangle-center}
\end{minipage}

\par
}%

The maximizer of the Hamiltonian is not of interest, since the Hamiltonian diverges to $+\infty$ when two vortices collide.
We therefore exclude collision configurations and restrict our discussion to the region where the Hamiltonian is finite.

The corresponding problem for three identical point vortices is much simpler, since the system is integrable.
The relative equilibria of the identical three vortex problem consist of equilateral and collinear configurations, where the equilateral configuration is the global minimum and the collinear configuration is a saddle point. (The global minimality of the equilateral configuration can be easily shown by the AM-GM inequality.)
Their classification and stability properties were developed in the works of Synge \cite{Synge1949}, Tavantzis and Ting \cite{TavantzisTing1988}, and Aref \cite{Aref2009}.

However, the four-vortex problem is not integrable, and the classification of its relative equilibria is a more challenging problem.
Determining the exact variational nature of these relative equilibria provides a fundamental step toward understanding the geometry of the underlying phase space.
Characterizing all stationary configurations yields important information about the structure of its high-dimensional energy landscape.
In particular, identifying the global minimizer is of special interest, as global optimization in such a high-dimensional nonlinear setting is generally a challenging problem.

\begin{remark} [Conjecture on five-vortex system]
It is natural to expect that the two configurations that Mayer found on five magnets are the only stable configurations in the five-vortex system.
The stability of the regular pentagon is well known, as we can see in the result of the Kurakin and Yodovich \cite{Kurakin2002}.
The square with a center configuration belongs to the family consisting of a regular $n$-gon together with a central vortex, whose relative equilibria and stability were analyzed by Cabral and Schmidt \cite{CabralSchmidt2000}.
In particular, the square with a center configuration is a stable relative equilibrium of the five-vortex system.
There may be other relative equilibria of the five-vortex system such as collinear configurations, but they are expected to be unstable saddle points.
\end{remark}

The distinctive contribution of this paper is twofold. First, we provide an independent, purely analytical classification of these configurations using a Hadamard-coordinate reduction, offering an elegant alternative to previous classifications that relied heavily on computational algebraic geometry. Second, and more importantly, we completely determine the variational nature of each configuration. By establishing the square as the unique global minimizer in the constrained Hamiltonian landscape, our results provide a rigorous mathematical justification for physical phenomena, such as the stable square patterns observed in Mayer's floating-magnet experiments.

\subsection{Previous works on vortex problems}
Relative equilibria are among the simplest coherent motions of a vortex system and have been studied from dynamical, variational, topological, and algebraic points of view.
 The regular vortex polygon is the best-known family. 
 The classical investigations of Thomson and Havelock initiated the study of its stability, and Kurakin and Yudovich proved nonlinear stability for the regular $n$-gon with identical circulations when $n\leq 7$ and instability when $n\geq 8$ \cite{Kurakin2002}. 
 Polygonal configurations with an additional central vortex were studied by Cabral and Schmidt, who analyzed their Lyapunov stability and bifurcations \cite{CabralSchmidt2000}. 
 This latter family includes, when $n=3$, the equilateral triangle with a vortex at its center that appears in our classification.

A complementary line of research treats relative equilibria as critical points of the Hamiltonian under conserved-quantity constraints. 
Palmore used critical-point and topological methods to obtain multiplicity results for positive circulations \cite{Palmore1982}. 
Such results establish the existence of many relative equilibria, but they do not by themselves provide an explicit classification for a fixed number of vortices. 
The present paper addresses this more concrete classification problem in the special case of four identical circulations.

The four-vortex problem has been studied directly by several algebraic and computational methods. 
O'Neil converted the equations for relative equilibria and collapse configurations into polynomial systems and obtained finiteness results and upper bounds for general circulations \cite{ONeil2007}. 
Hampton and Moeckel subsequently proved finiteness results for relative equilibria, absolute equilibria, and rigidly translating configurations using exact symbolic computation and algebraic geometry \cite{HamptonMoeckel2009}.
More recently, Yu established finiteness results for stationary configurations of the planar four-vortex problem, including rotating, translating, equilibrium, and collapse configurations \cite{Yu2023}.
This comprehensive result provided a definitive mathematical proof that the total number of any such configurations is strictly bounded for arbitrary circulations.
Our work pushes beyond finiteness to provide a complete and explicit geometric classification for the special case of identical vortices.

Several works have also exploited partial equality among the circulations. 
P\'erez-Chavela, Santoprete, and Tamayo studied symmetric relative equilibria when three circulations are equal, determining solution counts and bifurcation values as the fourth circulation varies \cite{PerezChavelaSantopreteTamayo2015}. 
Menezes and Roberts classified collinear relative equilibria with three equal circulations and analyzed their bifurcations and linear stability \cite{MenezesRoberts2018}. 
These studies show that equality of circulations creates additional symmetry that can substantially simplify the four-vortex equations. 
Recent work has further investigated the possible geometries of general four-vortex relative equilibria through a parameterized polynomial formulation \cite{KallyadanShukla2025}. 
In that broader setting, the circulations are part of the configuration problem, whereas here all four circulations are fixed and identical.

The specific case of four identical vorticities was previously analyzed as a subcase within a broader study by Hampton, Roberts, and Santoprete \cite{HamptonRobertsSantoprete2014}. 
They investigated the relative equilibria of the four-vortex problem where two pairs of vortices have equal strengths. 
Using techniques from computational algebraic geometry, such as Gr\"obner bases, they showed that when all four vorticities are equal, exactly three geometrically distinct configurations exist: a square, an equilateral triangle with a central vortex, and a collinear configuration. 
In contrast to their computational methods, our work introduces a purely mathematical approach by exploiting the special algebraic structure of the system, completely characterizing the variational nature of each configuration.

\subsection{Organization of the paper}

This paper is organized as follows. In Section 2, we narrow the problem using previous results.
We formulate the precise problem in Section 3.
In Section 4, we introduce the Hadamard transformation and reduce the four-vortex problem to an equivalent three-point problem.
We find all possible stationary configurations in Section 5 and characterize their energetic properties in Section 6.
Finally, we prove the main theorem in Section 7.

\section{Stationary Configurations of Identical Vortices}

In this section, we recall previous results on the classification of stationary configurations.
These results substantially reduce the problem and allow us to focus on relative equilibria.

\begin{proposition}\label{prop:sc_re}
Let $(z_1, z_2, z_3, z_4) \in \mathbb{C}^4$ be a configuration with identical intensities $\Gamma \neq 0$.
Then the configuration is a stationary configuration if and only if it is a relative equilibrium.
\end{proposition}

In \cite{ONeil1987}, O'Neil proved that stationary configurations in the $N$-vortex problem can be classified into four types: equilibria, rigidly translating configurations, relative equilibria, and collapse configurations.
The necessary conditions established by O'Neil are summarized
in Table \ref{tab:stationary_conditions}.
$\sigma = \sum_{i} \Gamma_i$ and $L = \sum_{i<j} \Gamma_i \Gamma_j$ are the total circulation and the total vortex angular momentum, respectively.

\begin{table}[htbp]
\centering
\renewcommand{\arraystretch}{1.3}
\begin{tabular}{|l|l|}
\hline
\textbf{Stationary Configuration Type} & \textbf{Necessary Condition} \\ \hline
Equilibria & $L = 0$ \\ \hline
Rigidly translating configurations & $\sigma = 0$ \\ \hline
Relative equilibria & None for $L$ or $\sigma$ (can exist for $L \neq 0, \sigma \neq 0$) \\ \hline
Collapse configurations & $L = 0$ \\ \hline
\end{tabular}
\caption{Necessary conditions for the four types of stationary configurations.}
\label{tab:stationary_conditions}
\end{table}

\begin{proof}[Proof of Proposition~\ref{prop:sc_re}]

When the intensities are identical, with $\Gamma \neq 0$, we have $\sigma = 4\Gamma$ and $L = 6\Gamma^2$, both of which are nonzero.
Thus, relative equilibria are the only possible type of stationary configuration in the identical-vortex system.
By the classification of stationary configurations of O'Neil, the stationary configurations of the identical-vortex system are exactly the relative equilibria.

\end{proof}

\section{Variational Formulation on the Reduced Phase Space}

\subsection{Relative equilibria as critical points of the Hamiltonian}

It is a classical result in Hamiltonian dynamics that relative equilibria correspond to critical points of the Hamiltonian restricted to the level sets of conserved quantities (see, e.g., \cite{Palmore1982}).
For completeness, we identify the Lagrange multipliers and establish a criterion for critical points on phase space.

Let $\mathbf{z} = (z_1, z_2, z_3, z_4) \in \mathbb{C}^4$ be a relative equilibrium with center $z_c \in \mathbb{C}$ and impulse $I \in \mathbb{R}^+$.
By translational invariance, we can assume without loss of generality that $z_c = 0$, so that $I = \sum_{i=1}^4 |z_i|^2$.
When all vortex intensities are identical, with $\Gamma \neq 0$, the Hamiltonian simplifies to

\[
W = -\frac{1}{2\pi} \Gamma^2 \sum_{1 \le i < j \le 4} \log |z_i - z_j|
\]
 
A relative equilibrium $\mathbf{z}$ (consider $\mathbf{z} = (x_1, x_2, x_3, x_4, y_1, y_2, y_3, y_4) \in \mathbb{R}^8$) rotates uniformly with constant angular velocity $\omega$, so that $\dot{\mathbf{z}} = \omega J \mathbf{z}$ where $J = \begin{pmatrix} 0 & -I_4 \\ I_4 & 0 \end{pmatrix}$.
Comparing this with the Kirchhoff equations $\Gamma \dot{\mathbf{z}} = - J \nabla W(\mathbf{z})$ and using $\nabla I(\mathbf{z}) = 2\mathbf{z}$, we obtain
\[
\nabla \left( W + \frac{1}{2}\Gamma \omega I \right) = 0.
\]
By the Lagrange multiplier theorem, a critical point $\mathbf{z} \in \mathcal{Z}$ satisfies $\nabla W = \lambda_1 \nabla I + \boldsymbol{\lambda}_c$ for some $\lambda_1 \in \mathbb{R}$ and $\boldsymbol{\lambda}_c \in \mathbb{R}^2$ associated with the center of the configuration constraints.
Because $W$ is translation-invariant, summing the components gives $\boldsymbol{\lambda}_c = 0$, leaving $\nabla W = 2\lambda_1 \mathbf{z}$ with $\displaystyle{\omega = - \frac{2}{\Gamma} \lambda_1}$.
So a configuration is a relative equilibrium if and only if it is a critical point of the Hamiltonian.

To compute the multiplier explicitly, let
\[
G_i := \sum_{j \neq i} \frac{z_i - z_j}{|z_i - z_j|^2}, \quad i = 1, 2, 3, 4.
\]
Then $\displaystyle{\nabla_{z_i} W = -\frac{\Gamma^2}{2\pi} G_i}$. 
The condition $\nabla W = 2\lambda_1 \mathbf{z}$ is equivalent to $\displaystyle{G_i = - \frac{4\pi \lambda_1}{\Gamma^2} z_i}$.
Taking the inner product with $z_i$ and summing over $i$, we have
\[
-\frac{4 \pi \lambda_1}{\Gamma^2} \sum_{i=1}^4 |z_i|^2 = \sum_{i=1}^4 z_i \cdot G_i = \sum_{1 \le i < j \le 4} \frac{|z_i - z_j|^2}{|z_i - z_j|^2} = 6.
\]
Since $\sum_{i=1}^4 |z_i|^2 = I$, we find $\displaystyle{\lambda_1 = -\frac{3 \Gamma^2}{2\pi I}}$ and the angular velocity $\displaystyle{\omega = \frac{3 \Gamma}{\pi I}}$.
Thus, we obtain the angular velocity of a relative equilibrium.

\begin{lemma}\label{lemma:angular-velocity}
Let $(z_1, z_2, z_3, z_4) \in \mathbb{C}^4$ be a relative equilibrium, with center $z_c \in \mathbb{C}$, impulse $I \in \mathbb{R}^+$, and identical intensity $\Gamma \neq 0$. 
Then the angular velocity of the configuration is $\displaystyle{\omega = \frac{3 \Gamma}{\pi I}}$.
\end{lemma}

\begin{proof}
The proof is a direct consequence of the above discussion.
\end{proof}

We also obtain a useful criterion for determining whether a
configuration is a relative equilibrium.
The lemma below is used to check candidate relative equilibria in later sections.

\begin{lemma}[Critical point criterion]\label{lemma:critical-point-criterion}
A configuration $\mathbf{z} = (z_1, z_2, z_3, z_4) \in \mathcal{Z}$ with center $0 \in \mathbb{C}$ and impulse $I \in \mathbb{R}^+$ is a relative equilibrium if and only if it satisfies
\[
G_i = \frac{6}{I} z_i, \quad i = 1, 2, 3, 4.
\]
\end{lemma}

\begin{proof}
From the above discussion, substituting $\displaystyle{\lambda_1 = -\frac{3 \Gamma^2}{2\pi I}}$ into $\displaystyle{G_i = - \frac{4 \pi \lambda_1}{\Gamma^2} z_i, \quad i = 1, 2, 3, 4}$, we have the desired result.
\end{proof}

\begin{remark}
In fact, the critical points of $W$ on $\mathcal{Z}$ occur in orbits.
Since we have not removed rotational symmetry in $\mathcal{Z}$, if $\mathbf{z}$ is a critical point of $W$, every rotation of $\mathbf{z}$ is also a critical point of $W$.
It is natural since the relative equilibrium is a configuration that rotates uniformly, so the orbit of the configuration is also a relative equilibrium.
\end{remark}

\subsection{Reduction of the phase space}

By Noether's theorem, the invariance of the point-vortex Hamiltonian under spatial translations and rotations guarantees the conservation of the center $\bar{z} = \frac{1}{4}\sum_{i=1}^4 z_i$ and the impulse $I = \sum_{i=1}^4 |z_i - \bar{z}|^2$, respectively.
Due to the translational and scaling symmetries of the equations of motion, any relative equilibrium can be mapped to a canonical one by a suitable spatial translation, a scaling transformation, and a time reparameterization. 
Therefore, without loss of generality, it is sufficient to set $\Gamma = 1$, $\bar{z} = 0$, and $I = 1$.

Under these normalizations, the search for relative equilibria is naturally restricted to the reduced phase space $\mathcal{Z}$, defined as

\[
\mathcal{Z} = \left\{ (z_1, z_2, z_3, z_4) \in \mathbb{C}^4 : \sum_{i=1}^4 z_i = 0, \sum_{i=1}^4 |z_i - \bar{z}|^2 = 1, |z_i - z_j| > 0 \text{ for all } i \neq j \right\}
\]

$|z_i - z_j| > 0 \text{ for all } i \neq j$ condition is added since otherwise the Hamiltonian diverges to $+\infty$, which is not of interest.
Once the relative equilibria in $\mathcal{Z}$ are classified, the complete set of solutions for arbitrary $\bar{z}, I$, and $\Gamma$ can be recovered by applying the corresponding inverse transformations.

\section{The Equivalent Three-point Problem via Hadamard Transformation}

We introduce a useful transformation using Hadamard matrix of order 4, which enables us to reduce the four-vortex problem to an equivalent three-point problem.
This transformation separates the implicit constraint from the problem.

\subsection{Hadamard transformation}

In this section, we introduce a useful transformation of the coordinates of the four vortices.
A Hadamard matrix of order 4 ($\mathbf{H}_4$) is a $4 \times 4$ matrix whose entries are either $+1$ or $-1$, and whose rows are mutually orthogonal.
Applying this linear transformation to $z_1, z_2, z_3, z_4$, we obtain a new set of coordinates.

\[
\begin{pmatrix}
0 \\
u \\
v \\
w
\end{pmatrix}
=
\mathbf{H}_4
\begin{pmatrix}
z_1 \\
z_2 \\
z_3 \\
z_4
\end{pmatrix}
=
\begin{pmatrix}
1 & 1 & 1 & 1 \\
1 & 1 & -1 & -1 \\
1 & -1 & 1 & -1 \\
1 & -1 & -1 & 1
\end{pmatrix}
\begin{pmatrix}
z_1 \\
z_2 \\
z_3 \\
z_4
\end{pmatrix}.
\]

The first coordinate is always 0 in $\mathcal{Z}$; thus we can isolate the constraint.
We can express $z_1, z_2, z_3, z_4$ in terms of $u, v, w$ as follows.

\[
z_1 = \frac{1}{4}(u + v + w), \quad
z_2 = \frac{1}{4}(u - v - w), \quad
z_3 = \frac{1}{4}(-u + v - w), \quad
z_4 = \frac{1}{4}(-u - v + w).
\]

We can now rewrite the constraints and the Hamiltonian in terms of $u,v,w$.
First, the constraint from the translational symmetry is automatically satisfied.
The normalization of the angular impulse can be rewritten as

\[
|u|^2 + |v|^2 + |w|^2 = \mathbf{z}^{*} \mathbf{H}_4^{*} \mathbf{H}_4 \mathbf{z} = 4 \mathbf{z}^{*} \mathbf{z} = 4(\sum_{i=1}^4 |z_i|^2) = 4.
\]

We used $\mathbf{H}_4^{*} \mathbf{H}_4 = 4 \mathbf{I}$.
Rewriting the Hamiltonian in terms of $u, v, w$, we have

\[
W = - \frac{1}{2\pi} \log \left( \frac{1}{64} |u^2 - v^2| |v^2 - w^2| |w^2 - u^2| \right).
\]

This expression is obtained by substituting the expressions for $z_1,z_2,z_3,z_4$ in terms of $u,v,w$.
The simplicity of the expression is also a crucial advantage of this transformation.

\subsection{Equivalent three-point formulation}

With the above transformation, we reduce the problem to a three-point problem.
Set $a=u^2$, $b=v^2$, and $c=w^2$. The three points $a,b,c \in \mathbb{C}$ satisfy the constraint $|a|+|b|+|c|=4$.
Thus, the Hamiltonian $W$ and the state space $\mathcal{M}$ can be rewritten as follows.

\[
W = - \frac{1}{2\pi} \log \left( \frac{1}{64} |a-b| |b-c| |c-a| \right)
\]

\[
\mathcal{M} = \left\{ (a, b, c) \in \mathbb{C}^3 : |a|+|b|+|c| = 4, |a-b||b-c||c-a| \neq 0 \right\}
\]

The problem is now converted to a three-point problem, which makes the analysis much easier.
We work primarily with $\Pi := |a-b| |b-c| |c-a|$ instead of $W$.
The critical point discussion of $\Pi$ is equivalent to that of $W$ in $\mathcal{M}$.

\begin{table}[ht]
    \centering
    \renewcommand{\arraystretch}{2}
    \begin{tabularx}{\textwidth}{
        |>{\centering\arraybackslash}X
        |>{\centering\arraybackslash}X|
    }
        \hline
        $\mathcal{Z}$ & $\mathcal{M}$ \\
        \hline

        Permutation of
        $(z_1, z_2, z_3, z_4)$
        &
        Permutation of
        $(a, b, c)$
        \\
        \hline

        Global rotation of $\phi$
        &
        Global rotation of $2\phi$
        \\
        \hline

        Reflection about the $x$-axis
        &
        Reflection about the $x$-axis
        \\
        \hline
    \end{tabularx}

    \caption{Symmetry correspondence table}
    \label{tab:symmetry-correspondence}
\end{table}

Table \ref{tab:symmetry-correspondence} shows how the symmetries of $\mathcal{Z}$ correspond to those of the three-point problem in $\mathcal{M}$.
The conversion of the permutation can be understood via the factor group isomorphism of $S_4 / V_4 \cong S_3$, where $V_4$ is the Klein four-group.
The Hamiltonian in $\mathcal{M}$ has the same rotational and reflection symmetries as in $\mathcal{Z}$.

We can transfer the critical-point condition from the original formulation.
The following lemma establishes this transfer using the diffeomorphism defined by the transformation.

\begin{lemma}\label{lemma:diffeomorphism}

Let $(z_1, z_2, z_3, z_4) \in \mathcal{Z}$ and $(a, b, c) \in \mathcal{M}$ be the corresponding point.
If $abc \neq 0$, $(z_1, z_2, z_3, z_4)$ is a critical point of $W$ if and only if $(a, b, c)$ is a critical point of $\Pi$ in $\mathcal{M}$.

\end{lemma}

\begin{proof}

Under the Hadamard transformation
\[
(0, u, v, w)^t = H_4(z_1, z_2, z_3, z_4)^t,
\]

define

\[
F: \mathcal{Z} \to \mathcal{M}, \qquad
F(z_1, z_2, z_3, z_4) = (u^2, v^2, w^2) = (a, b, c).
\]

If $abc \neq 0$, equivalently $uvw \neq 0$, then the Jacobian matrix of $F$ is invertible.
Thus, by the inverse function theorem, $F$ is a local diffeomorphism. 
Therefore, the critical point condition of $W$ on $\mathcal Z$ is equivalent to the that of $\Pi$ on $\mathcal M$.

\end{proof}

\section{Determination of Relative Equilibria}

We investigate two cases separately: $abc=0$ and $abc\neq 0$.
By Lemma \ref{lemma:diffeomorphism}, finding critical points of $\Pi$ in $\mathcal{M}$ is equivalent to finding critical points of $W$ in $\mathcal{Z}$ when $abc \neq 0$. 

\subsection{The case $abc = 0$: collinear and square configurations}

Without loss of generality, we can assume $b = 0, a \neq 0, c \neq 0$.
There is only one zero among $a,b,c$, since otherwise two points would coincide.
(If we set $a = 0$ or $c = 0$ instead of $b = 0$, this only permutes the points $(z_1, z_2, z_3, z_4)$ and gives the same result.)
Then we have $|a| + |c| = 4$ and $\Pi = |a||c||a-c|$.
From $b=0$, we obtain $v=0$, which gives the configuration a strong symmetry.
$v = z_1 - z_2 + z_3 - z_4 = 0$ and the constraint $z_1 + z_2 + z_3 + z_4 = 0$ give $z_1 + z_3 = z_2 + z_4 = 0$.

We focus on the reduced state spaces under $b=0$, which are

\[
\mathcal{Z'} := \left\{(z_1, z_2, z_3, z_4) \in \mathbb{C}^4 : z_1 + z_3 = 0, z_2 + z_4 = 0, \sum_{i=1}^4 |z_i|^2 = 1, |z_i - z_j| > 0 \text{ for all } i \neq j \right\}
\]

\[
\mathcal{M'} := \left\{(a, c) \in \mathbb{C}^2 : |a| + |c| = 4, |a||c||a - c| > 0 \right\}
\]

Since $\mathcal{Z}' \subset \mathcal{Z}$,

\[
d(W|_{\mathcal{Z}})_\mathbf{z} = 0 \implies d(W|_{\mathcal{Z'}})_\mathbf{z} = 0,
\]

We first find the critical points of $W$ on $\mathcal{Z'}$ and then check whether they are also critical points of $W$ on $\mathcal{Z}$.

As in Lemma \ref{lemma:diffeomorphism}, we obtain a local diffeomorphism $F': \mathcal{Z'} \to \mathcal{M'}$, defined by $F'(z_1,z_2,z_3,z_4)=(u^2,w^2)=(a,c)$.
Thus, we can transfer the critical-point condition for $W$ on $\mathcal{Z'}$ to that for $\Pi$ on $\mathcal{M'}$.

Let $|a| = r \in (0, 2]$, and let the angle between $a$ and $c$ be $\theta \in [0, \pi]$ so that $(a, c) = (r, (4-r)e^{i \theta})$.
(We assume $|a| \leq |c|$ and $a \in \mathbb{R}^+$ without loss of generality, since this only permutes the points or rotates the configuration.)
Then $\Pi = r (4-r) \sqrt{r^2 + (4-r)^2 - 2 r (4-r) \cos \theta}$.
By taking the derivative of $\Pi$ with respect to $\theta$,

\[
\frac{\partial \Pi}{\partial \theta} = r (4-r) \frac{r (4-r) \sin \theta}{\sqrt{r^2 + (4-r)^2 - 2 r (4-r) \cos \theta}} = 0
\]

and we get $\sin \theta = 0$, which gives $\theta = 0$ or $\theta = \pi$.

\paragraph{(1) $\theta = 0$}\mbox{}\\

If $\theta = 0$, then $\Pi = r (4-r) |r - (4-r)| = r (4-r) (4-2r)$.
Taking the derivative of $\Pi$ with respect to $r$,

\[
\frac{d\Pi}{dr} = 6r^2 - 24r + 16 = 0
\]

Thus we get $r = 2 - \frac{2}{\sqrt{3}}$.
This gives a stationary collinear configuration $\{(x_1, 0), (x_2, 0), (-x_1, 0), (-x_2, 0)\}$ in $\mathcal{Z}'$, where $\displaystyle{x_1 = \frac{1}{2\sqrt{3-\sqrt{6}}}}$ and $\displaystyle{x_2 = \frac{1}{2\sqrt{3+\sqrt{6}}}}$.
We get $G_i = 6 z_i$ for $i = 1, 2, 3, 4$, and thus this configuration is indeed a critical point of $W$ on $\mathcal{Z}$, by Lemma \ref{lemma:critical-point-criterion}.
The plot appears in Figure \ref{fig:Collinear}.

\begin{figure}[h]
    \centering
    \begin{tikzpicture}[scale=3]
        \draw[->] (-0.9,0) -- (0.9,0) node[right] {$x$};
        \draw[->] (0,-0.9) -- (0,0.9) node[above] {$y$};
        
        \draw ( 0.5, 0.02) -- ( 0.5,-0.02) node[below] {$\tfrac12$};
        \draw (-0.5, 0.02) -- (-0.5,-0.02) node[below] {$-\tfrac12$};
    
        \draw ( 0.02, 0.5) -- (-0.02, 0.5) node[left] {$\tfrac12$};
        \draw ( 0.02,-0.5) -- (-0.02,-0.5) node[left] {$-\tfrac12$};

        \pgfmathsetmacro{\xone}{
          0.25*(
            sqrt(2 + 2/sqrt(3))
            +
            sqrt(2 - 2/sqrt(3))
          )
        }

        \pgfmathsetmacro{\xtwo}{
          0.25*(
            sqrt(2 + 2/sqrt(3))
            -
            sqrt(2 - 2/sqrt(3))
          )
        }

        \foreach \x in {\xone,\xtwo,-\xone,-\xtwo}{
          \fill (\x,0) circle[radius=0.9pt];
        }

        \draw (-1.1, -1.1) rectangle (1.1, 1.1);
    \end{tikzpicture}
    \caption{Collinear configuration.}
    \label{fig:Collinear}
\end{figure}
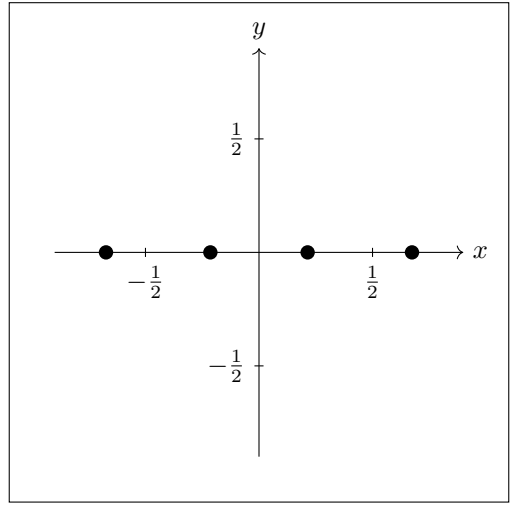

\paragraph{(2) $\theta = \pi$}\mbox{}\\

If $\theta = \pi$, then $\Pi = r (4-r) |r + (4-r)| = 4 r (4-r)$.
Taking derivative of $\Pi$ with respect to $r$,

\[
\frac{\partial \Pi}{\partial r} = 16 - 8r = 0
\]

and we get $r = 2$.
This gives a stationary square configuration $\{ (\frac{1}{2}, 0), (0, \frac{1}{2}), (-\frac{1}{2}, 0), (0, -\frac{1}{2}) \}$.
We get $G_i = 6 z_i$ for $i = 1, 2, 3, 4$, and thus this configuration is indeed a critical point of $W$ on $\mathcal{Z}$, by Lemma \ref{lemma:critical-point-criterion}.
The plot appears in Figure \ref{fig:Square}.

\begin{figure}[h]
    \centering
    \begin{tikzpicture}[scale=3]
        \draw[->] (-0.9,0) -- (0.9,0) node[right] {$x$};
        \draw[->] (0,-0.9) -- (0,0.9) node[above] {$y$};
        
        \draw ( 0.5, 0.02) -- ( 0.5,-0.02) node[below] {$\tfrac12$};
        \draw (-0.5, 0.02) -- (-0.5,-0.02) node[below] {$-\tfrac12$};
    
        \draw ( 0.02, 0.5) -- (-0.02, 0.5) node[left] {$\tfrac12$};
        \draw ( 0.02,-0.5) -- (-0.02,-0.5) node[left] {$-\tfrac12$};

        \fill ( 0.5, 0) circle[radius=0.9pt];
        \fill (-0.5, 0) circle[radius=0.9pt];
        \fill (0,  0.5) circle[radius=0.9pt];
        \fill (0, -0.5) circle[radius=0.9pt];

        \draw (-1.1, -1.1) rectangle (1.1, 1.1);
    \end{tikzpicture}
    \caption{Square configuration.}
    \label{fig:Square}
\end{figure}
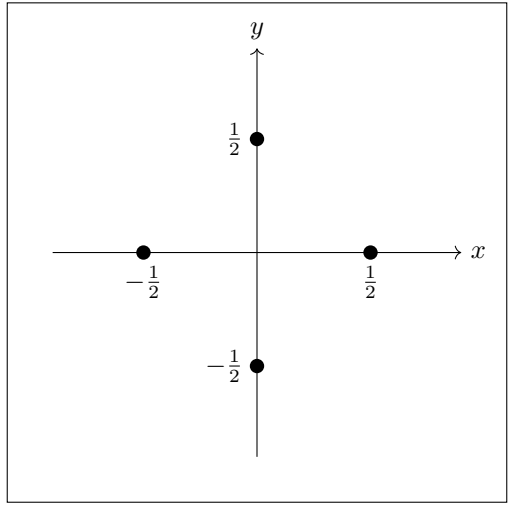

\subsection{The case $abc \neq 0$: equilateral triangle with a center}

\begin{lemma}\label{lemma:stationary_s}

Let $S(s) = |a - s| + |b - s| + |c - s|$ be a function of $s \in \mathbb{C}$, which is the sum of distances from $s$ to three points $a, b, c \in \mathbb{C}$.
If $(a, b, c) \in \mathcal{M}$ is a critical point which satisfy $abc \neq 0$, then $s = 0$ is a critical point of $S(s)$, i.e., $dS_0 = 0$.

\end{lemma}

\begin{proof}

We construct a path in $\mathcal{M}$ passing through $(a,b,c)$ using translation.
Consider a slight translation of the three points, $(a, b, c) \mapsto (a - s, b - s, c - s)$ for some $s \in \mathbb{C}$.
Since $\Pi$ solely depends on the distances between points, the translation does not change $\Pi$.
But this translation does not satisfy the constraint $|a| + |b| + |c| = 4$, so we need to scale the points by a factor of $ \displaystyle{\lambda(s) := \frac{4}{|a - s| + |b - s| + |c - s|} = \frac{4}{S(s)}}$.
This gives a path in $\mathcal{M}$ passing through $(a,b,c)$ in the direction $s$:

\[
t \mapsto f(t) = (\lambda(ts)(a - ts), \lambda(ts)(b - ts), \lambda(ts)(c - ts))
\]

where $t \in \mathbb{R}$.
If we set $\Pi_0 = \Pi(a, b, c)$, then each point in the path has $\Pi(f(t)) = \lambda(ts)^3 \Pi_0$.

Since $(a,b,c)$ is a critical point, the derivative of $\Pi$ along any path passing through $(a,b,c)$ is 0.
Thus, we have $\displaystyle{\left. \frac{d\Pi(f(t))}{dt} \right\vert_{t=0} = 0}$. Then

\[
\left. \frac{d\Pi(f(t))}{dt} \right|_{t = 0}
=
\left. \frac{d}{dt} \left( \lambda(ts)^3\Pi_0 \right) \right|_{t = 0}
=
3 \Pi_0 d\lambda_0(s)
= 0.
\]

Since $\Pi_0 \neq 0$, we have $d\lambda_0(s) = 0$. 
Using $\displaystyle{\lambda(s) = \frac{4}{S(s)}}$ and $S(s) \neq 0$, we obtain $dS_0(s)=0$.
The above argument is valid for every direction $s$, so $s=0$ is a critical point of $S$, i.e., $dS_0=0$.

\end{proof}

Using Lemma \ref{lemma:stationary_s}, we obtain a necessary condition for critical points when $abc \neq 0$.
$dS_0 = 0$ then $\nabla S(0) = 0$, which is written as

\[
\nabla S(0) = - \left( \frac{a}{|a|} + \frac{b}{|b|} + \frac{c}{|c|} \right) = 0.
\]

This means that the sum of the unit vectors from the origin to the three points is zero. Without loss of generality, we can assume that $a, b, c$ have arguments $0, \frac{2\pi}{3}, \frac{4\pi}{3}$, respectively.
Let $|a| = \alpha, |b| = \beta, |c| = \gamma$, where $\alpha, \beta, \gamma \in (0, 4)$ and $\alpha + \beta + \gamma = 4$.
Then $\Pi^2$ can be written as

\begin{align*}
\Pi^2
&= (\alpha^2 + \alpha \beta + \beta^2)(\beta^2 + \beta \gamma + \gamma^2)(\gamma^2 + \gamma \alpha + \alpha^2) \\
&= s_1^2 s_2^2 - s_2^3 - s_1^3 s_3 \\
&= 16s_2^2 - s_2^3 - 64s_3.
\end{align*}

where $s_1 = \alpha + \beta + \gamma = 4$, $s_2 = \alpha \beta + \beta \gamma + \gamma \alpha$, and $s_3 = \alpha \beta \gamma$.
Using the Lagrange multiplier method subject to the constraint $\alpha + \beta + \gamma = 4$, we get

\[
\begin{cases}
(\alpha - \beta)(A - 64 \gamma) = 0 \\
(\beta - \gamma)(A - 64 \alpha) = 0 \\
(\gamma - \alpha)(A - 64 \beta) = 0
\end{cases}
\]

where $A = 32 s_2 - 3 s_2^2$.
The three values $\alpha, \beta, \gamma$ cannot all be distinct, since then $A = 64 \alpha = 64 \beta = 64 \gamma$, which is a contradiction.
Thus we have two cases: (1) $\alpha = \beta \neq \gamma$ and (2) $\alpha = \beta = \gamma$.

\paragraph{(1) $\alpha = \beta \neq \gamma$}\mbox{}\\

Let $\alpha = \beta = x$ and $\gamma = 4 - 2x$, where $x \in (0, 2)$.
Here we get $A - 64x = 32 s_2 - 3 s_2^2 - 64x = 0$, $s_2 = -3x^2 + 8x$ and hence

\[
A - 64x = -3x (3x - 4) (3x^2 - 12x + 16) = 0
\]

$x$ cannot be 0, and the quadratic equation $3x^2 - 12x + 16 = 0$ has no real solution.
So we get $x = \frac{4}{3}$, which gives $\alpha = \beta = \gamma = \frac{4}{3}$.
This contradicts the assumption $\alpha = \beta \neq \gamma$, so there is no valid solution in this case.

\paragraph{(2) $\alpha = \beta = \gamma$}\mbox{}\\

We directly get $\alpha = \beta = \gamma = \frac{4}{3}$, which gives a stationary configuration $\{(0, 0), (\frac{1}{\sqrt{3}}, 0), (-\frac{1}{2\sqrt{3}}, \frac{1}{2}), (-\frac{1}{2\sqrt{3}}, -\frac{1}{2})\}$ in $\mathcal{Z}$.
We get $G_i = 6 z_i$ for $i = 1, 2, 3, 4$, and thus this configuration is indeed a critical point of $W$ on $\mathcal{Z}$, by Lemma \ref{lemma:critical-point-criterion}.
The plot appears in Figure \ref{fig:TriangleCenter}.

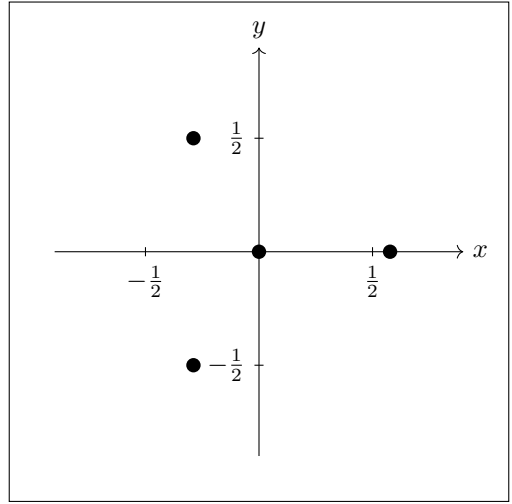
\begin{figure}[h]
    \centering
    \begin{tikzpicture}[scale=3]
        \draw[->] (-0.9,0) -- (0.9,0) node[right] {$x$};
        \draw[->] (0,-0.9) -- (0,0.9) node[above] {$y$};
        
        \draw ( 0.5, 0.02) -- ( 0.5,-0.02) node[below] {$\tfrac12$};
        \draw (-0.5, 0.02) -- (-0.5,-0.02) node[below] {$-\tfrac12$};
    
        \draw ( 0.02, 0.5) -- (-0.02, 0.5) node[left] {$\tfrac12$};
        \draw ( 0.02,-0.5) -- (-0.02,-0.5) node[left] {$-\tfrac12$};

        \pgfmathsetmacro{\rightx}{1/sqrt(3)}
        \pgfmathsetmacro{\leftx}{-1/(2*sqrt(3))}

        \fill (0,0) circle[radius=0.9pt];
        \fill (\rightx,0) circle[radius=0.9pt];
        \fill (\leftx,0.5) circle[radius=0.9pt];
        \fill (\leftx,-0.5) circle[radius=0.9pt];

        \draw (-1.1, -1.1) rectangle (1.1, 1.1);
    \end{tikzpicture}
    \caption{Equilateral triangle with a center configuration.}
    \label{fig:TriangleCenter}
\end{figure}

\section{Variational Characterization of Relative Equilibria}

So far, we have found all possible relative equilibria of the four-vortex system.
We now examine the nature of each configuration in the Hamiltonian landscape.
By showing that the minimum of the Hamiltonian exists and must be a critical point, we can conclude that the square configuration is the unique global minimizer of the Hamiltonian.
For the other two configurations, we show that they are saddle points of the Hamiltonian by finding two directions along which the Hamiltonian increases and decreases.

\subsection{The global minimizer}

\begin{proposition}\label{prop:global_minimizer}
The square configuration, which has the smallest Hamiltonian value among all relative equilibria, is the unique global minimizer of the Hamiltonian $W$ in $\mathcal{Z}$.
\end{proposition}

\begin{lemma}\label{lemma:compact_manifold}
Let 

\[
\mathcal{X} = \left\{ (z_1, z_2, z_3, z_4) \in \mathbb{C}^4 : \sum_{i=1}^4 z_i = 0, \sum_{i=1}^4 |z_i|^2 = 1 \right\}.
\]

Then $\mathcal{X}$ is a smooth and compact manifold.
\end{lemma}

\begin{proof}
For the smoothness, we first consider the space without normalization,
\[
\mathcal{L} := \{\mathbf{z} = (z_1, z_2, z_3, z_4) \in \mathbb{C}^4 : z_1 + z_2 + z_3 + z_4 = 0\}.
\]
$\mathcal{L}$ is a complex linear subspace of $\mathbb{C}^4$ of complex dimension $3$, which is isomorphic to $\mathbb{C}^3$ (or equivalently, to $\mathbb{R}^6$).
For a smooth function $I(\mathbf{z}) = \sum_{i=1}^4 |z_i|^2$ defined on $\mathcal{L}$, we have $\mathcal{X} = I^{-1}(1)$.
It is enough to show that $1$ is a regular value of $I$.
For $\mathbf{v} \in T_{\mathbf{z}} \mathcal{L} = \mathcal{L}$, we have

\[
dI_{\mathbf{z}}(\mathbf{v})
= 2 \operatorname{Re} \sum_{i=1}^4 \overline{z_i} v_i
\]

If $\mathbf{z} \in I^{-1}(1)$, then $dI_{\mathbf{z}}(\mathbf{z}) = 2 \sum_{i=1}^4 |z_i|^2 = 2$.
Thus $dI_{\mathbf{z}}$ is surjective for every $\mathbf{z} \in I^{-1}(1)$.
Hence, by the regular value theorem, $\mathcal{X}$ is a smooth embedded submanifold of $\mathcal{L}$ of real dimension $\dim_{\mathbb{R}} (\mathcal{X}) = \dim_{\mathbb{R}} \mathcal{L} - 1 = 6 - 1 = 5$.
In fact $\mathcal{X}$ is diffeomorphic to the $5$-dimensional sphere $S^5$.

$\mathcal{X} = I^{-1}(1)$ is closed in $\mathcal{L}$, and $\mathcal{L}$ is closed in $\mathbb{C}^4$, so $\mathcal{X}$ is closed in $\mathbb{C}^4$.
It is also bounded since $\sum_{i=1}^4 |z_i|^2 = 1$.
Thus $\mathcal{X}$ is compact by Heine-Borel theorem.
\end{proof}

The smoothness and compactness of $\mathcal{X}$ are used in following lemmas.

\begin{lemma}[Existence of a minimizer]\label{lemma:existence_of_minimizer}

There exists $\mathbf{z}^* \in \mathcal{Z}$ that minimizes the Hamiltonian $W$ in $\mathcal{Z}$.
\end{lemma}

\begin{proof}
Since $\Pi$ is continuous on $\mathcal{X}$ and $\mathcal{X}$ is compact, there exists a maximizer $\mathbf{z}^* \in \mathcal{X}$ by the extreme value theorem.
It is easy to see that $\Pi(\mathbf{z}^*) > 0$, since for instance, $\Pi$ has a value of 16 at the square configuration.
(Recall that we defined $\Pi(\mathbf{z}) = |a - b| |b - c| |c - a| = 64 \prod_{i<j} |z_i - z_j|$.)
So we can write $\mathbf{z}^* \in \mathcal{X} \setminus \Pi^{-1}(0) = \mathcal{Z}$.
In $\mathcal{Z}$, $W$ is well-defined and continuous. 
In $W = - \frac{1}{2\pi} \log \left( \frac{1}{64} \Pi \right)$, the $-\log$ is a strictly decreasing function, so $\mathbf{z}^*$ minimizes $W$ in $\mathcal{Z}$. 
\end{proof}

\begin{lemma}\label{lemma:minmizer_is_critical}
If $\mathbf{z}^* \in \mathcal{Z}$ minimizes the Hamiltonian $W$ in $\mathcal{Z}$, then $\mathbf{z}^*$ is a critical point.
\end{lemma}

\begin{proof}
Since $\mathcal{Z}$ is an open subset of the smooth manifold $\mathcal{X}$, it is itself a smooth manifold, and $W$ is smooth on it.

$\mathbf{z}^*$ is a global minimizer of $W$ on $\mathcal{Z}$, so it is in particular a local minimizer. Hence, by Fermat's theorem on smooth manifolds,

\[
d\left(W|_{\mathcal{Z}}\right)_{\mathbf{z}^*}=0.
\]

Therefore, $\mathbf{z}^*$ is a critical point of $W$ on $\mathcal{Z}$.
\end{proof}

\begin{proof}[Proof of Proposition~\ref{prop:global_minimizer}]
By Lemmas \ref{lemma:existence_of_minimizer} and \ref{lemma:minmizer_is_critical}, a minimizer of $W$ in $\mathcal{Z}$ must be one of the critical points.
The subscripts $s$, $c$, and $e$ denote square, collinear, and an equilateral triangle with a center configurations, respectively.
$\Pi_s = 16$, $\displaystyle{\Pi_c = \frac{32}{3\sqrt{3}}}$, and $\displaystyle{\Pi_e = \frac{64}{3\sqrt{3}}}$, so $\Pi_s > \Pi_e > \Pi_c$.
Then $W_s < W_e < W_c$, so the square configuration is a global minimizer of $W$ in $\mathcal{Z}$.
\end{proof}

\subsection{Saddle points}

\begin{proposition}\label{prop:saddle_points}
The collinear configuration and the equilateral triangle with a center configuration are saddle points of the Hamiltonian $W$ in $\mathcal{Z}$.
\end{proposition}

\begin{proof}
To show that a configuration is a saddle point of $W$ in $\mathcal{Z}$, it suffices to find two paths in $\mathcal{Z}$ passing through the configuration, such that $W$ is strictly increasing along one path and strictly decreasing along the other path.

For the collinear configuration, we can parametrize a path in reduced space $\mathcal{M}'$ as $(a, c) = (r, (4-r)e^{i \theta})$ with $r \in (0, 2]$ and $\theta \in [0, \pi]$.
Then 

\[
\Pi(r, \theta) = r (4-r) \sqrt{r^2 + (4-r)^2 - 2 r (4-r) \cos \theta}.
\]

The collinear configuration corresponds to $(r, \theta) = (r_0, 0)$, where $r_0 = 2 - \frac{2}{\sqrt{3}}$.
First, we keep $\theta = 0$ and vary $r$.
Then 
\[
\Pi = r (4-r) (4-2r) =: g(r)
\]
and we get $g'(r) = 6r^2 - 24r + 16$ and $g''(r) = 12r - 24$.
Since $g'(r_0) = 0$ and $g''(r_0) = 12 r_0 - 24 < 0$, $\Pi(r, 0) < \Pi(r_0, 0)$ for every $r \neq r_0$ near $r_0$.

Second, we keep $r = r_0$ and vary $\theta$.
Then
\[
\Pi = r_0 (4-r_0) \sqrt{(4-2r_0)^2 + 2r_0 (4-r_0)(1-\cos{\theta})}.
\]
$1-\cos{\theta}$ has a local minimum at $\theta = 0$, so $\Pi(r_0, \theta) > \Pi(r_0, 0)$ for every $\theta \neq 0$ near $0$.
Thus, we have found two paths in $\mathcal{M}'$ passing through the collinear configuration, such that $\Pi$ is strictly increasing along one path and strictly decreasing along the other path.
Hence, the collinear configuration is a saddle point of $\Pi$ in $\mathcal{M}'$.
By the local diffeomorphism $F': \mathcal{Z'} \to \mathcal{M'}$, the collinear configuration is a saddle point of $W$ in $\mathcal{Z'}$ and hence in $\mathcal{Z}$.

For the equilateral triangle with a center configuration, we can write the point in $\mathcal{M}$ as $(a_0, b_0, c_0) = (\rho, \rho \omega, \rho \omega^2)$ with $\rho = \frac{4}{3}$ and $\omega = e^{i \frac{2\pi}{3}}$.
The first path is $t \mapsto (a_t, b_t, c_t) = (\rho + t, (\rho - t) \omega, \rho \omega^2)$ where $t \in \mathbb{R}$ and $|t| < \rho$.
Then
\[
\Pi(a_t, b_t, c_t)^2 = |a_t - b_t|^2 |b_t - c_t|^2 |c_t - a_t|^2  = \frac{4096}{27} + t^6
\]
and this is strictly greater than $\Pi(a_0, b_0, c_0)^2 = \frac{4096}{27}$ for every $t \neq 0$.

The second path is similar to the path we used in Lemma \ref{lemma:stationary_s}.
We construct $t \mapsto h(t) = (\lambda(t)(a_0 - t), \lambda(t)(b_0 - t), \lambda(t)(c_0 - t))$ in $\mathcal{M}$ where $\lambda(t) = \frac{4}{S(t)} = \frac{4}{|a_0 - t| + |b_0 - t| + |c_0 - t|}$ where $t \in \mathbb{R}$ and $|t| < \rho$.
For sufficiently small $t \neq 0$, it is easy to see that $S(t) > 3\rho = 4$. 
This gives $\lambda(t) < 1$ and hence $\Pi(h(t)) = \lambda(t)^3\Pi(a_0, b_0, c_0) < \Pi(a_0, b_0, c_0)$, so $\Pi$ strictly decreases along the path $h(t)$. 
Thus, the equilateral triangle with a center configuration is a saddle point of $\Pi$ in $\mathcal{M}$. 
By the local diffeomorphism $F : \mathcal{Z} \to \mathcal{M}$, the equilateral triangle with a center configuration is a saddle point of $W$ in $\mathcal{Z}$.

\end{proof}

\section{Proof of the Main Theorem}

\begin{proof}
By Proposition \ref{prop:sc_re}, it suffices to find all relative equilibria in order to classify all stationary
configurations of the four-identical-vortex system.

We reduce the phase space to $\mathcal{Z}$ and then extend the discussion to the full phase space.
Thus, the problem is formulated as finding critical points of the Hamiltonian on $\mathcal{Z}$.

Using the Hadamard transformation, we convert the original problem into a three-point problem.
We find solutions in the degenerate and nondegenerate cases, corresponding to $abc=0$ and $abc\neq 0$, respectively, yielding a total of three possible stationary configurations: a square, a collinear configuration, and an equilateral triangle with a center.
Propositions \ref{prop:global_minimizer} and
\ref{prop:saddle_points} determine the nature of each
configuration in the Hamiltonian landscape.
Specifically, the square configuration is a global minimizer of the Hamiltonian.

We now return to general configurations in $\mathbb{C}^4$.
After the appropriate spatial translation and scaling, a stationary configuration has center $z_c \in \mathbb{C}$, impulse $I \in \mathbb{R}^+$, and identical intensity $\Gamma \neq 0$.
By Lemma \ref{lemma:angular-velocity}, its angular velocity is $\displaystyle{\omega = \frac{3 \Gamma}{\pi I}}$.
Hence, we obtain the three types of stationary configurations in $\mathbb{C}^4$ listed in the theorem, completing the proof.

\end{proof}

This result can also explain the observation of Mayer in the floating magnet experiment.
Our result shows that the square configuration is the only stable configuration.
There are three possible stationary configurations, but with perturbations in the real world, every configuration may eventually reach the square configuration.
This aligns with the Mayer's experiment.

The effectiveness of the reduction relies on the special algebraic structure of four points and, in particular, on the existence of the Hadamard matrix of order four. 
For this reason, the method in its present form is specific to the four-vortex problem. 
Nevertheless, the calculation illustrates how discrete symmetry and an appropriate choice of internal coordinates can simplify the classification and variational analysis of relative equilibria.

\section*{Acknowledgements}

The author would like to thank Prof. In-Jee Jeong for proposing this research problem and for many helpful discussions throughout this work.

This work was supported by Yonsei University under BRL for Kinetic Dynamics and Continuum Mechanics (Project RS-2024-00406821).

Claude Fable 5 (Anthropic) was used during the initial exploration of a Hadamard-transformation-based reduction of the four-point problem to an equivalent three-point problem. 
The authors independently verified and developed all mathematical arguments presented in the manuscript and take full responsibility for its content.

\bibliographystyle{plain}
\bibliography{references}

\end{document}